\documentclass[twoside,a4page]{article}
\usepackage{bbm}
\usepackage{amsmath}
\usepackage{amsthm}
\usepackage{enumitem}
\usepackage{multicol}
\usepackage{hyperref}
\usepackage{amssymb}
\usepackage{nicefrac}

\renewcommand{\mathbb}[1]{\mathbbm{#1}}
\newcommand{\keyword}[1]{\textbf{#1}}
\newcommand{\N}{\mathbb{N}}
\newcommand{\Z}{\mathbb{Z}}
\newcommand{\id}{\mathrm{id}}
\newcommand{\R}{\mathbb{R}}

\newcommand{\supp}{\mathop{\mathrm{supp}}}
\newcommand{\after}{\mathbin{\circ}}

\newcounter{main}
\theoremstyle{plain}
\newtheorem{lem}[main]{Lemma}
\newtheorem{thm}[main]{Theorem}

\newtheorem{prop}[main]{Proposition}
\newtheorem{cla}[main]{Claim}
\theoremstyle{definition}
\newtheorem{dfn}[main]{Definition}
\newtheorem{exa}[main]{Example}

\title{Yosida duality}

\author{Bas Westerbaan \\ 
{\small Radboud Universiteit Nijmegen} \\
{\small \texttt{bas@westerbaan.name}}}

\begin{document}

\maketitle

\begin{abstract}
    In this note we prove Yosida duality ---
        that is: the category
        of compact Hausdorff spaces with continuous maps
        is dually equivalent
        to the category of uniformly complete Archimedean Riesz spaces
            with distinguished units and unit-preserving Riesz homomorphisms
            between them.
\end{abstract}

For a compact Hausdorff space~$X$, write~$C(X)$
    for the vector space of real-valued continuous functions on it.
The space~$C(X)$ contains sufficient additional structure to reconstruct
$X$ from it.  In fact, there are several ways to do it.

One particularly well-known method is using Gel'fand-duality:
    it states that any commutative unital C$^*$-algebra
    is isomorphic to the algebra of
    \emph{complex valued} continuous functions on
        some up-to-isomorphism unique compact Hausdorff spaces.
The categorical corollary is that the category
    of unital C$^*$-algebras with unital $*$-homomorphisms
    is dually equivalent to that of compact Hausdorff spaces.

The proof of Gel'fand's Theorem is rather involved ---
    in this note we will present a lesser known method
    due to Yosida\cite{yosida} which will lead to a different duality.
Gel'fand uses the product on~$C(X)$
    to reconstruct~$X$.  Yosida instead uses the lattice order.
At the end of this note,
    we will have shown that the category
        of compact Hausdorff spaces
        is dually equivalent
        to the category of uniformly complete Archimedean unitary Riesz spaces.

The present proof is based
    on a lecture series given by A.~van Rooij in 2011
    of which notes can be found here\cite{riesz}.
The present text takes a short path to Yosida duality;
    for a broad treatment of Riesz spaces,
    see e.g.~\cite{riesz,rooij,zaanen}.

\section{Riesz spaces}
We will first define Riesz spaces and derive their elementary theory.
\begin{dfn}
\begin{enumerate}
    \item
A real \keyword{ordered vector space}~$E$ is
    a real vector space with a partial order~$\leq$ such
    that for all~$x,y\in E$ with~$x \leq y$ we have
\begin{enumerate}
\item $x + a \leq y+a$ for any~$a\in E$ and
\item $\lambda \cdot x \leq \lambda \cdot y$ for all scalars $\lambda \geq 0$.  
\end{enumerate}
\item
A \keyword{Riesz space} is a real ordered vector space that is a lattice
(i.e.~all finite infima and suprema exist.)
\item
A linear map~$\varphi \colon E \to E'$ between Riesz spaces
is called a \keyword{Riesz homomorphism}
if~$\varphi$ preserves finite infima and suprema.
\end{enumerate}
\end{dfn}
\begin{exa}
    \begin{enumerate}
    \item 
For any topological space~$X$,
    the vector space~$C(X)$ of continuous real-valued functions on~$X$ is
    a Riesz space.
\item
    Let~$\leq_{\mathrm{lex}}$ be the lexicographic order on~$\mathbb{R}^2$ ---
    that is
        \begin{equation*}
            (x,y) \leq_{\mathrm{lex}} (x',y') \quad \iff \quad
            x < x' \text{ or } (x=x' \text{ and } y \leq y').
        \end{equation*}
    Then~$\mathbb{R}^2$ ordered by~$\leq_{\mathrm{lex}}$
        is a Riesz space.
    \end{enumerate}
\end{exa}
In a Riesz space
the lattice operations interact nicely with
other structure:
\begin{lem}\label{lem:basics}
    For a Riesz space~$E$ with elements~$x,y,a \in E$ it holds
    \begin{enumerate}
        \item $\lambda (x \vee y) = (\lambda x) \vee (\lambda y)$
                for scalars $\lambda \geq 0$
        \item $-(x \vee y) = (-x) \wedge (-y)$
        \item $(x+a) \vee (y+a) = (x \vee y) + a$
        \item $x \vee y + x \wedge y = x+y$
        \item $(x \wedge y)\vee a = (x \vee a) \wedge (y \vee a)$
    \end{enumerate}
\end{lem}
\begin{proof}
    The maps~$x \mapsto \lambda x$ and~$x \mapsto x+a$ are order isomorphisms
        (if~$\lambda > 0$)
        and these preserve suprema, hence~1 and 3.
    Negation is an order anti-isomorphism and so 2.
    Point 4 follows from the previous two:
    \begin{equation*}
    (x \vee y) - x - y
    \stackrel{3}{=} (x - x - y) \vee ( y- x- y) = (-y)\vee (-x)
            \stackrel{2}{=} - (x \wedge y).
    \end{equation*}
    In every lattice~$(x \wedge y) \vee a \leq (x \vee a) \wedge (y \vee a)$.
    For the other inequality, note
    \begin{equation*}
    (x \vee a) \wedge (y \vee a) \leq x \vee a
    \stackrel{4}{=} x+a - x\wedge a \leq x+a - x\wedge y \wedge a
    \end{equation*}
    and so~$(x \vee a)\wedge(y \vee a) + x\wedge y \wedge a -a \leq x$.
    Similarly~$(x \vee a)\wedge(y \vee a) + x\wedge y \wedge a -a \leq y$.
    And so~$(x \vee a)\wedge(y \vee a) + x\wedge y \wedge a -a \leq x\wedge y$.
    \begin{equation*}
    (x \vee a)\wedge(y \vee a)
        \leq x \wedge y + a - x\wedge y \wedge a
        \stackrel{4}{=} (x\wedge y) \vee a. \qedhere
    \end{equation*}
\end{proof}
The lattice-order on a Riesz space gives quite some structure.
\begin{dfn}
For any Riesz space~$E$, define
\begin{enumerate}
\item
    the \keyword{absolute value} $|x| := x \vee (-x)$ for $x,y \in E$;
\item
    the \keyword{positive cone} $E^+ := \{ e \in E; e \geq 0\}$;
\item
    \keyword{orthogonality} $x \perp y :\Leftrightarrow |x| \wedge |y| = 0$
    for~$x,y\in E$ \emph{and}
\item
    the \keyword{positive} 
        $x^+ := x \vee 0$ and \keyword{negative parts}
        $x^- := (-x) \vee 0$.
\end{enumerate}
\end{dfn}
Note that in~$C(X)$ the absolute value is point-wise;
    $C(X)^+$ are the positive valued functions
    \emph{and} $f \perp g$ iff $f$ and~$g$ have disjoint support.
\begin{lem}
    For any Riesz space~$E$ with~$x,y \in E$ we have
    \begin{multicols}{2}
    \begin{enumerate}
        \item $x = x^+ - x^-$
        \item $|x| = x^+ + x^-$
        \item $|x| \geq 0$
        \item $|x| = x^+ \vee x^-$
        \item $x^+ \perp x^-$
        \item $|\lambda x| = |\lambda||x|$
        \item $|x|=0 \implies x=0$
        \item $(x+y)^+ \leq x^+ + y^+$
        \item $|x+y| \leq |x| + |y|$
        % \item $x-x \wedge y \perp y - x\wedge y$
        \item $x \vee y = \frac{x+y}{2} + \left|\frac{x-y}{2}\right|$
    \end{enumerate}
    \end{multicols}
\end{lem}
\begin{proof}
    As~$x = x + 0 = x \vee 0 + x \wedge 0 =
                    x \vee 0 - (-x) \vee 0 =
                x^+ - x^-$ we have 1. So
    \begin{equation*}
    x^+ + x^- = 2x^+ - x = 2(x \vee 0) - x
        =   (2x) \vee 0 - x = (2x-x) \vee (-x)
        = |x|
    \end{equation*}
    hence 2.
    Clearly~$0 \leq 0 \vee x = x^+$ and similarly~$0 \leq x^-$.
        So by the previous point ~$0 \leq x^+ + x^- = |x|$
            hence 3.
    So~$|x| = |x| \vee 0 = x \vee (-x) \vee 0
                =(x \vee 0) \vee ((-x) \vee 0)
                = x^+ \vee x^-$, which is point 4.
    Also~$x^+ + x^- = x^+ \vee x^- + x^+ \wedge x^-$
        thus we must have~$x^+ \wedge x^- = 0$ viz point 5.
    For~$\lambda \geq 0$
        we have~$|\lambda| |x| = \lambda (x \vee (-x))
        = (\lambda x) \vee (- \lambda x) = |\lambda x|$
        and as~$|x| = |-x|$
        we establish 6.
    Note~$-|x| \leq x \leq |x|$ and so
        if~$|x| =0$ we have~$x=0$ i.e.~point 7.
    Always~$x \leq x^+$.
    So~$x+y \leq x^+ + y^+$.
    Hence~$(x+y)^+ \leq (x^++y^+)^+ = x^++y^+$ so 8.
    Similarly~$(x+y)^- \leq x^- + x^-$.
    So 9 follows:~$|x+y| \leq |x| + |y|$.
    From Lemma~\ref{lem:basics} point~3 it follows
    \begin{equation*}
        \begin{split}
        x+y + \left|x-y\right|
        &= x+y +( x-y) \vee (y-x)\\
        & = (x+y+x-y) \vee (x+y+y-x)
        = 2(x \vee y),
        \end{split}
    \end{equation*}
    which (together with pt.~2) shows pt.~10 and completes the proof.
    % To prove 10, first note~$x \wedge y \leq x$
    %     so~$x-x\wedge y \geq 0$.
    % Hence~$
    %     |x- x\wedge y | \wedge |y - x\wedge y|
    %     = (x- x\wedge y ) \wedge (y - x\wedge y)
    %     =(x \wedge y) - (x \wedge y) = 0$.
\end{proof}
\begin{lem}[Riesz Decomposition Lemma]
    Let~$E$ be any Riesz space.
   For any positive~$x,a,b \in E^+$ with~$x \leq a+b$
    there is a decomposition~$x = a'+b'$
    with~$0 \leq a' \leq a$ and  $0 \leq b' \leq b$.
\end{lem}
\begin{proof}
    Define~$a' := x \wedge a$ and~$b' := x - x \wedge a$.
    Clearly~$x = a'+b'$ and~$a' \leq a$.
    We have to show~$b' \leq b$.
    Note~$b' = x - x \wedge a = x \vee a - a$.
    By assumption~$x \leq a+b$
        and so~$x \vee a \leq (a+b)\vee a = a+b$
        hence~$b' \equiv x \vee a - a \leq b$.
\end{proof}

\section{Riesz subspaces and ideals}
To construct Riesz homomorphisms into~$\R$
    we study Riesz ideals.
\begin{dfn}
Let~$E$ be a Riesz space. A subset~$D\subseteq E$ is called
a \keyword{Riesz ideal} if it is the kernel of some Riesz homomorphism.
The ideal~$D$ is said to be \keyword{proper}
    when~$D \neq E$.
The set~$D \subseteq E$ is called a \keyword{Riesz subspace}
    if~$x\vee y, x\wedge y \in D$ whenever~$x,y\in D$.
\end{dfn}
\begin{lem}
    A linear subspace~$D \subseteq E$ is a Riesz subspace
        if and only if~$x^+ \in D$ whenever~$x \in D$.
\end{lem}
\begin{proof}
    If~$D$ is a Riesz subspace, then clearly~$x^+ = x \vee 0 \in D$
        whenever~$x \in D$.
    The converse follows
    from~$x \vee y = \frac{x+y}{2} + \left|\frac{x-y}{2}\right|$
        and~$|x| = x^+ + (-x)^+$.
\end{proof}
\begin{lem}\label{lem:rieszideal}
Given a Riesz space~$E$ and a~$D \subseteq E$, then
TFAE:
\begin{enumerate}
\item
$D$ is a Riesz ideal.
\item
$D$ is a linear subspace and for all~$a \in D$ and~$x \in E$:
\[ \text{if~$|x| \leq |a|$, then~$x \in D$. } \]
\item
$D$ is a Riesz subspace and for all~$a \in D^+$ and~$x \in E^+$:
\[ \text{if~$x \leq a$, then~$x \in D$. } \]
\end{enumerate}
\end{lem}
\begin{proof}
    1 $\Rightarrow$ 2.\quad  Assume~$D$ is a Riesz ideal,
        say it is the kernel of a Riesz homomorphism~$f$.
    Clearly~$D$ is a linear subspace.
        Assume~$|x| \leq |a|$ for some~$x\in E$ and~$a \in D$.
        Note Riesz homomorphisms preserve absolute
            values as~$|x|\equiv x\vee(-x)$ and
            so~$|f(x)| = f(|x|) \leq f(|a|) = |f(a)| = 0$
        hence~$f(x) = 0$. Thus~$x \in D$, as desired.

    2 $\Rightarrow$ 3.\quad Assume~$D$ is a linear subspace
        for which~$x\in D$ whenever~$|x| \leq |a|$ for some~$a \in D$.
    We only need to show~$D$ is a Riesz subspace.
    Assume~$a \in D$.  We have to show~$a^+ \in D$.
        Recall~$a^+ \leq a^+ +a^- = |a|$
            and so~$a^+ \in D$, as desired.

    3 $\Rightarrow$ 1.\quad
        Consider the linear quotient~$q\colon E \to E/_D$.
        We will turn~$E/_D$ into a Riesz space.
        We want to use~$q(E^+)$ as positive cone for a partial order on~$E/_D$;
        that is we define~$q(x) \leq q(y) \iff q(y)-q(x) \in q(E^+)$.
        This turns~$E/_D$ into a ordered vector space
            if~$q(E^+)$ is closed under addition, scalar multiplication
                by positive reals and if we
                    have~$q(a),-q(a) \in q(E^+) \Rightarrow q(a)=q(0)$.
        Clearly~$q(E^+)$ is closed under addition and scalar multiplication
            by positive reals.
         Note~$q(a) \in q(E^+)$ iff there is~$d \in D$ with~$d \leq a$.
        To show the last condition, assume~$q(a), -q(a) \in q(E^+)$.
        Then~$d \leq a$ and~$d' \leq -a$ for some~$d,d' \in D$.
        Hence~$0 \leq a -d \leq -d' -d$.
        By assumption on~$D$ we find~$a-d \in D$.
        Thus~$d \in D$ and~$q(a)=q(0)$.
        Indeed~$q(E^+)$ is a positive cone and we can
            order~$E/_D$ using it.
            Automatically~$q$ is order preserving.

    To show~$E/_D$ is a Riesz space
        assume~$q(x),q(y) \in E/_D$.
            Clearly~$q(x \wedge y) \leq q(x)$
                and~$q(x \wedge y) \leq q(y)$.
    Assume there is a~$q(a)\in E/_D$
        with~$q(a) \leq q(x)$ and~$q(a) \leq q(y)$.
    Then~$x - c \geq d$ 
        and~$y - c \geq d'$ for some~$d,d' \in D$.
    Consequently~$x\wedge y - c = (x - c) \wedge (y-c) \geq d \wedge d'$.
    As~$d\wedge d' \in D$ we see~$q(c)\leq q(x \wedge y)$.
    We have shown~$E/_D$ has pairwise infima.
    As~$x \mapsto -x$ is an order anti-isomorpism
        we see~$E/_D$ has pairwise suprema as well.
    We also saw~$q$ is a Riesz homomorphism
        and by construction~$D$ is its kernel.
\end{proof}
\begin{lem}\label{lem:perpideal}
    For any~$a \in E$ the set~$a^\perp := \{ x;\ x\in E;\  x \perp a\}$
        is a Riesz ideal.
\end{lem}
\begin{proof}
    We will first prove that~$a^\perp$ is a linear subspace.
    Let~$x\in a^\perp$
        and~$\lambda \in \R$.
    Assume~$|\lambda| \leq 1 $.
    Then~$0 \leq |\lambda x| \wedge |a| = |\lambda||x| \wedge |a| \leq |x| \wedge |a| = 0$
    and so~$\lambda x \perp  a$.
    In the other case~$|\lambda| \geq n$ for some~$n \in \N$ with~$n \neq 0$
        and so~$|\frac{\lambda}{n}| \leq 1$.
        By the previous~$|\frac{\lambda}{n}| \wedge |a| = 0$.
        Hence~$0= |\lambda| \wedge (n|a|) \geq |\lambda|\wedge |a|$, as desired.

    Let~$x,y \in a^\perp$.
    By distributivity~$(|x| \vee |y|) \wedge |a|
                = (|x|\wedge |a|) \vee (|y| \wedge |a|) = 0$
                so~$|x| \vee |y| \perp a$.
                By the previous~$2 (|x| \vee |y|) \perp a$.
    Thus
    \begin{equation*}
    |x +y|\wedge |a| \leq (|x|+|y|)\wedge |a|
    \leq (2(|x| \vee |y|)) \wedge |a| = 0.
    \end{equation*}
    Thus~$a^\perp$ is a linear subspace.
    To show~$a^\perp$ is a Riesz ideal,
        we will demonstrate condition 3 from Lemma~\ref{lem:rieszideal}.
    Assume~$0 \leq x \leq y \in a^\perp$.
    Then~$0 \leq x \wedge |a| \leq y \wedge |a| = 0$
        and so~$x \in a^\perp$ --- it remains to be
        shown that~$a^\perp$ is a Riesz subspace.
    It is sufficient to show~$x^+ \in a^\perp$ whenever~$x \in a^\perp$.
        Indeed from
        \begin{equation*}
        0 = |x|\wedge |a| = (x^+ \vee x^-) \wedge |a|
        = ( x^+ \wedge |a| ) \vee (x^- \wedge |a|)
        \end{equation*}
        it follows~$x^+ \perp a$ as desired.
\end{proof}

\section{Units and their norms}

\begin{dfn}
\begin{enumerate}
    \item
An element~$u \in E^+$ is called an (order) \keyword{unit}
        if for every~$x\in E$ there
exists a~$n\in \N$ such that~$|x| < n\cdot u$. A Riesz space is called
\keyword{unitary} if it has a unit.
\item
An element~$\varepsilon \in E$ is called \keyword{infinitesimal} if there is
a~$b \in E$ such that~$n\cdot \varepsilon \leq b$ for all~$n \in \Z$.
A Riesz space is called \keyword{Archimedean} if its only infinitesimal is~$0$.
\end{enumerate}
\end{dfn}

\begin{dfn}
Given an Riesz space~$E$ with unit~$u \in E$
    write
    \begin{equation*}
     \| x \|_u := \inf \{\lambda;\ \lambda \in \R^+;\  |x| \leq \lambda u\} .
    \end{equation*}
\end{dfn}
\begin{lem}
    Let~$E$ be a Riesz space with units~$e,u$.
    \begin{enumerate}
        \item $\|\  \|_u$ is a seminorm.
        \item If $\| x \|_u = 0$, then~$x$ is infinitesimal.
        \item If~$E$ is Archimedean,
                then $\|\  \|_u$ is a norm
                and~$|x| \leq \|x\|_u u$.
        \item $\|\ \|_e$ and~$\|\ \|_u$ are equivalent.
    \end{enumerate}
\end{lem}
\begin{proof}
    First we'll show~$\|\ \|_u$ is a seminorm.
    Let~$\varepsilon > 0$ be given.
    For any~$x \in E$ we have~$|x| \leq (\| x \|_u + \varepsilon) u$.
    So
    $ |a+b| \leq |a| + |b|
        \leq (\|a\|_u + \|b\|_u + 2\varepsilon) u$.
    Hence~$\|a+b\|_u \leq \|a\|_u + \|b\|_u + 2\varepsilon$.
    Thus, as~$\varepsilon>0$ was arbitrary, the triangle inequality
        holds for~$\|\ \|_u$.

    Now we'll show the absolute homogeneity.
    Let~$\mu > 0$ be given.
    With~$|\mu a|=|\mu||a|$
    and similar reasoning as before
        we see~$\|\mu a\|_u \leq |\mu| \|a\|_u$.
    To prove the other inequality
        first note that for any~$\|a\|_u \geq  \varepsilon > 0$
        we have~$|a| \nleq (\|a\|_u - \varepsilon)u$.
    Thus~$|\mu a | = |\mu| |a| \nleq (|\mu| \|a\|_u - |\mu|\varepsilon )u$.
    As we can make~$|\mu|\varepsilon$ arbitrarily small
        we see~$\|\mu a\|_u \geq |\mu| \|a\|_u$.
    This gives absolute homogeneity except for the case~$\mu=0$,
        which is trivially true.
    Finally, by definition~$\|\ \|_u$ is positive,
        so it's a seminorm.

    Assume~$\|x\|_u = 0$. Then~$|x| \leq \frac{1}{n} u$ for any~$n \in \N$
        with~$n >0$.
        Hence~$n |x| \leq u$
    Now, for any~$m \in \mathbb{Z}$
        we have~$m x \leq |mx| = |m| |x| \leq u$
        and so~$x$ is infinitesimal.

    If~$E$ is Archimedean then directly by the previous
        we see~$\|\ \|_u$ is a norm.
        Now we show~$|a| \leq \|a\|_u u$.
        Suppose~$n \in \N$ with~$n \neq 0$.
        Then~$|a| \leq (\|a\|_u + \frac{1}{n})u = \|a\|_u u + \frac{1}{n} u$.
        By taking the infimum over~$n \in \N_{>0}$,
            we see it suffices to
            show~$\inf \{ \frac{1}{n} u;\  n \in \N_{ > 0}\} = 0$.
        Clearly~$0$ is a lower-bound.
        Assume~$b \in E$ is another lower-bound,
            i.e.~$b \leq \frac{1}{n} u$ for all~$n \in \N_{>0}$.
        Then~$b^+ \leq (\frac{1}{n}u)^+ = \frac{1}{n}u$
            so~$n b^+ \leq \frac{1}{n}u$
                and~$-nb^+ \leq 0 \leq \frac{1}{n}u$
                for all~$n \in \N$.
        Thus~$b^+$ is infinitesimal: so~$b^+ = 0$
            and~$b \leq 0$, as desired.

        Finally we show point 4.
        As for any~$\varepsilon > 0$ we have~$|x| \leq (\|x\|_u + \varepsilon) u$
        we get~$\|x\|_e \leq (\|x\|_u + \varepsilon) \|u\|_e$
            and so $\|x\|_e \leq \|x\|_u  \|u\|_e$.
        Similarly$\|x\|_u \leq \|x\|_e \|e\|_u$
        and so~$\frac{1}{\|e\|_u}\|x\|_u \leq \|x\|_e \leq \|u\|_e \|x\|_u$,
        which shows~$\|\ \|_u$ and~$\|\ \|_e$ are equivalent norms.
\end{proof}

\begin{lem} \label{riesz-hom-are-cont}
Any non-zero Riesz homomorphism~$f\colon E \to E'$
between unital Archimedean Riesz spaces
is continuous with respect to the norms
induced by any unit.
\end{lem}
\begin{proof}
    Pick any unit~$u \in E$ and~$u' \in E'$.
    As all unit norms are equivalent
        it is sufficient to show~$f$ is continuous
        with respect to~$\|\ \|_u$ and~$\|\ \|_{u'}$.
    If~$\| f(u) \|_{u'} =0$, then~$f(u)=0$
        and so~$f=0$ as~$-\| x\|_u u \leq x \leq \|x\|_uu$
        for any~$x \in E$.  Thus~$\|f(u)\|_{u'} > 0$.
    Let~$\varepsilon > 0$ be given.
        If~$\|x - y\|_u \leq \delta$
            for some~$\delta>0$
        then~$|x-y| \leq \delta u$
            and so~$|f(x) - f(y)| = f(|x-y|) \leq \delta f(u)
                    = \delta |f(u)| \leq \delta \|f(u)\|_{u'} u'$.
    Thus it is sufficient to
        choose~$\delta \leq \frac{\varepsilon}{\|f(u)\|_{u'}}$.
\end{proof}
\begin{dfn}
An Archimedean unitary Riesz space is called \keyword{uniformly complete}
if~$\|\ \|_u$ is complete for some unit~$u \in E$.
\end{dfn}
As order-unit norms are equivalent we know that
    if~$E$ is complete with respect to one unit, then it is complete
    with respect to all units.
\begin{dfn}
We write~$\mathsf{CAURiesz}$ for the category of
uniformly complete
Archimedean Riesz spaces with distinguished unit as objects
and Riesz homomorphisms that preserve the selected unit as arrows between.
\end{dfn}

\section{Maximal Riesz ideals}
The real numbers form a totally ordered
Archimedean Riesz space and is, in fact, the only one.
This Proposition will play the same r\^ole
as the Banach--Mazur Theorem in the development
of Gel'fand Duality.
\begin{prop}
If~$E$ is a Riesz space with more than one element, then
TFAE.
\begin{enumerate}
\item
$E \cong \R$
\item
The only ideals of~$E$ are~$\{0\}$ and~$E$.
\item
$E$ is totally ordered and Archimedean.
\end{enumerate}
\end{prop}
\begin{proof}
1 $\Rightarrow$ 2.\quad Direct.

2 $\Rightarrow$ 3.\quad First we will prove~$E$ is totally ordered.
Assume $x,y\in E$. Then $J = \{v\in E; \ v \perp (x - y)^+ \}$ is
a Riesz ideal by Lemma~\ref{lem:perpideal}.
If~$J=\{0\}$ then since~$(x-y)^- \perp (x-y)^+$, we know
$(x-y)^-=0$ and thus~$0\leq x-y$. Hence~$x\geq y$.
In the other case: if~$J=E$, then $(x-y)^+ \perp (x-y)^+$ and thus~$(x-y)^+=0$.
Hence~$x-y = -(x-y)^- \leq 0$ and thus~$x \leq y$.

Now we will prove~$E$ is Archimedean.
Assume~$\varepsilon$ is an infinitesimal of~$E$.
We will show~$\varepsilon = 0$.
Let~$b \in E$ such that~$n |\varepsilon| \leq b$ for all~$n \in \N$.
Note~$J:=\{ x \in E; \ n |x| \leq b \text{ for all } n \in \N \}$
    is a Riesz ideal.
If~$J = \{0\}$ we are done.
In the other case~$J=E$.
Then~$b \in J$ and~$n|b| \leq b \leq n|b|$ for all~$n \neq 0$.
Thus~$2|b| = b = |b|$ so~$|\varepsilon| \leq b = 0$ and we are done.

3 $\Rightarrow$ 1.\quad
Pick any~$a > 0$.
The map~$\lambda \mapsto \lambda a$
    is a linear order isomorphism of~$\R$
    onto a $1$-dimensional subspace of~$E$.
We are done if this subspace is~$E$ itself.
Let~$b \in E^+$.
Define~$r := \sup\{ r;\ r \in \R^+;\ rb \leq a\}$.
As~$E$ is Archimedean,
    there is a~$n$ with~$nb \nleq a$.
    Thus~$nb \geq a$ and so~$r < \infty$.
For any~$\varepsilon > 0$
    we must have~$(r+ \varepsilon)b > a$ as~$E$ is totally ordered.
Taking infimum over~$\varepsilon> 0$ we see~$a = rb$
    thus indeed~$E$ is 1-dimensional.
\end{proof}

An important corollary is the following:
\begin{prop} \label{r-embedding-prop}
Given an Archimedean Riesz
space~$E$ with unit~$e$.
For every~$a \in E$, $a \neq 0$, there is a
Riesz homomorphism~$\varphi\colon E \to \R$ such that~$\varphi(e)=1$
and~$\varphi(a) \neq 0$.
\end{prop}
\begin{proof}
It is sufficient to find such~$\varphi$ for~$|a|$
as~$0\neq \varphi(|a|)=|\varphi(a)|$ implies~$\varphi(a) \neq 0$.
Thus assume~$a > 0$.
    As~$E$ is Archimedean, there is an~$n \in \N$
        with~$na \nleq e$.
    That is~$na - e \nleq 0$. But~$na -e \leq (na-e)^+$,
            so~$(na-e)^+ \neq 0$.

Consider the proper Riesz ideal~$J := \{ x;\ x \perp (na-e)^+\}$.
Find using Zorn's Lemma a maximal proper Riesz ideal~$J_0 \supseteq J$.
Then~$E/_{J_0}$ contains precisely two ideals and must be isomorphic to~$\R$.
Define~$\varphi_0$ as the composition of this isomorphism with
the quotient-map~$E  \to E/_{J_0} \cong \R$.
Note~$\varphi_0(e) > 0$
    as otherwise~$\varphi_0 = 0$.
Define~$\varphi := \frac{\varphi_0}{\varphi_0(e)}$.
Clearly~$\varphi(e)=1$.

Note~$(na -e)^- \in J$ and so~$0 = \varphi((na-e)^-)
    = (n\varphi(a) -\varphi(e))^-$
    which implies~$n\varphi(a)
        - \varphi(e) = (n \varphi(a)- \varphi(a))^+ \geq 0$.
Thus~$n \varphi(a) \geq \varphi(e) > 0$
    and so~$\varphi(a) \neq 0$.
\end{proof}

\section{Function spaces}
\begin{exa}\label{exa:CXisCAU}
    Let~$X$ be a compact Hausdorff space. Then~$C(X)$ with the
obvious pointwise operations and order, is a Riesz space. Since~$X$ is
compact, all these functions are bounded,
$\mathbb{1}(x):=1$ is a unit and
$\|\ \|_{\mathbb{1}}$ is a supnorm for which~$C(X)$ is complete.
It is also obvious that the zero function is the only
infinitesimal. Thus~$C(X)$ is Archimedean.
Thus~$C(X)$ with unit~$\mathbb{1}$ is an object of~$\mathsf{CAURiesz}$.
\end{exa}
Let~$X$ be a compact Hausdorff space
    and~$x \in X$.
    Then~\emph{point evaluation}
        $\delta_x\colon C(X) \to \R$ defined by~$\delta_x(f) = f(x)$
        is a Riesz homomorphism into~$\mathbb{R}$
        with~$\delta_x(\mathbb{1})=1$.
        Every unit-preserving Riesz homomorphism~$C(X) \to \R$
            is of this form:
\begin{prop} \label{phic-are-points}
Let~$X$ be a compact Hausdorff space.
For any Riesz homomorphism~$\varphi \colon C(X) \to \R$
    with~$\varphi(\mathbb{1})=1$
    there is a point~$x \in X$
    with~$\varphi = \delta_x$.
\end{prop}
\begin{proof}
    For any~$f \in C(X)$
        define~$\tilde{f} := | f - \varphi(f)\mathbb{1}|$.
    If we can find an~$x_0 \in X$
    with~$x_0 \notin \supp \tilde{f} := \{x;\ \tilde{f}(x) \neq 0 \}$
        for all~$f \in C(X)$,
        then we are done
            as~$\tilde{f}(x_0)=0$
            implies~$\varphi(f) = f(x_0) = \delta_{x_0}(f)$.
Note that by definition~$\tilde{f} \geq 0$
                and~$\varphi(\tilde{f}) = 0$.
    For a single~$f \in C(X)$ there must be an~$x\in X$
    with~$x \notin \supp \tilde{f}$ (i.e.~$\tilde{f}(x) = 0$)
        for otherwise~$\tilde{f} > 0$ and
        so by compactness~$\tilde{f} > \varepsilon \mathbb{1}$
            for some~$\varepsilon > 0$
            which gives~$0 = \varphi(\tilde{f}) > \varepsilon$ quod non.
    Now, let~$f_1,\ldots, f_n \in C(X)$ be given.
    Write~$g := \tilde{f}_1 \vee \ldots \vee \tilde{f}_n$.
    Note~$\varphi(g) = 0$
        and so there must also be an~$x \in X$ with~$g(x) = 0$.
As~$\supp g = \supp  \tilde{f}_1 \cup \ldots \cup \supp \tilde{f}_n$
        we see~$x \notin \supp \tilde{f}_i$ for every~$1 \leq i \leq n$.
Thus no finite subset of the
family of open sets~$\mathcal{A} := \{ \supp \tilde{f};\ f \in C(X)  \}$ can
        cover~$X$. Thus by contraposition of compactness we
        see that~$\cup \mathcal{A}$ cannot cover~$X$.
    Hence there is an~$x_0 \in X$ with~$x_0 \notin \supp \tilde{f}$ for
    any~$f\in C(X)$ as desired.
\end{proof}
This Proposition is a cornerstone of the duality. Before
we will look into that, we will proof a variant
of the Stone-Weierstrass theorem in the language of Riesz spaces.
\begin{thm}[Stone-Weierstrass]\label{thm:stoneweierstrass}
For any compact Hausdorff space~$X$ and Riesz subspace~$D$
of~$C(X)$ such that:
\begin{enumerate}
\item
    $D$ is unital --- that is: $\mathbb{1} \in D$
\item
    $D$ separates the points --- that is:
        for all~$x, y\in X$ with~$x\neq y$
        there is an~$f \in D$ with~$f(x) \neq f(y)$.
\end{enumerate}
Then:~$D$ is~$\|\ \|_{\mathbb{1}}$-dense in~$C(X)$.
\end{thm}
\begin{proof}
    Let~$g \in C(X)$ and~$\varepsilon > 0$ be given.
    We will find a~$g_0 \in D$
    with~$\| g-g_0\|_{\mathbb{1}} \leq \varepsilon$.
    By approximating~$g^+$ and~$-g^-$ separately
        we see we may assume without loss of generality~$g \geq 0$.
    We will approximate~$g$ progressively in three steps.

    Let~$y, z \in X$ be any points with~$y \neq z$.
    By assumption there is an~$f \in D$ with~$f(z) \neq f(y)$.
    Define~$f_y := \frac{g(z)}{f(z)-f(y)}(f - f(y) \mathbb{1})$.
    Note~$f_y(y)=0$, $f_y(z)=g(z)$ and still~$f_y \in D$.

    Write~$U_{f_y} := \{ x;\ x \in X;\ f_y(x) < g(x) +
        \frac{\varepsilon}{2}\}$.
    Clearly both~$z, y \in U_{f_y}$.
    Thus the~$U_{f_y}$ are an open cover of~$X$.
    By compactness we can find~$n \in \N$ and~$y_1, \ldots, y_n$
        such that~$U_{y_1} \cup \cdots \cup U_{y_n} = X$.
    Define~$f'_{z} := f_{y_1} \wedge \cdots \wedge f_{y_n}$.
    The map~$f'_{z} \in D$ is a better approximation:
    $f'_{z} (z) = g(z)$
    and~$f'_{z} \leq g + \frac{\varepsilon}{2} \mathbb{1}$.

    Write~$V_{f'_z} := \{x;\  g(x) - \frac{\varepsilon}{2} < f'_z(x) \}$,
        which are also open sets with~$z \in V_{f'_z}$.
    By compactness we can find~$m \in \N$
        and~$z_1, \ldots, z_m$
        with~$V_{f'_{z_1}} \cup \cdots \cup V_{f'_{z_m}} = X$.
        Now define~$g_0 := f'_{z_1} \vee \cdots \vee f'_{z_m}$.
        Note~$g - \frac{\varepsilon}{2} \mathbb{1} \leq  g_0
        \leq g + \frac{\varepsilon}{2} \mathbb{1} $
        and so~$\|g_0 - g\|_{\mathbb{1}} \leq \varepsilon$.
\end{proof}
\section{The spectrum}
We can turn a compact Hausdorff space~$X$
into an Archimedean uniformly complete
Riesz space~$C(X)$ --- how do we go the other way?
\begin{dfn}
    Let~$E$ be an Archimedean Riesz space with a unit~$u$.
    The \keyword{spectrum} of~$E$ is defined to be
    the set
    \begin{equation*}
    \Phi(E) := \{ \varphi;\ \varphi\colon E \to \mathbb{R} \ 
    \text{Riesz homomorphism with } \varphi(u)=1 \}.
    \end{equation*}
    which is a topological space with induced
        topology of the inclusion~$\Phi(E) \subseteq  \R^E$
            where we take the product topology on the latter.
\end{dfn}
At first glance, the spectrum seems to depend on the choice of unit. However:
\begin{cla}
    Let~$E$ be an Archimedean Riesz space with units~$e,u \in E$.
    We will write~$\Phi_e(E)$ and~$\Phi_u(E)$ for the spectra
        respective to~$e$ and~$u$.\\
    Then: $\Phi_e(E) \cong \Phi_u(E)$
    via~$\varphi \mapsto (x \mapsto \frac{\varphi(x)}{\varphi(u)})$.
\end{cla}
\begin{lem}\label{lem:specishsdf}
    The spectrum~$\Phi(E)$ is a compact Hausdorff space.
\end{lem}
\begin{proof}
    As for any~$x \in E$
    it holds~$-\|x\|_u \leq \varphi(x) \leq \|x\|_u$
    we have the following inclusions
        of topological spaces
        \begin{equation*}
            \Phi(E) \ \subseteq\  \prod_{x \in E} [-\|x\|_u, \|x\|_u]
                    \ \subseteq\  \R^E.
        \end{equation*}
    The middle one is compact by Tychonoff's Theorem theorem.
    Thus if we can show~$\Phi(E)$
    is closed in~$\R^E$, we know~$\Phi(E)$ is compact and Hausdorff.

    Let~$(\varphi_\alpha)_\alpha$ in~$\Phi(E)$ be a net
        converging to some~$\varphi$ in~$\R^E$.
    As addition,  infimum, supremum and scalar multiplication
        are continuous the map~$\varphi$ must be a Riesz homorphism as well.
        Trivially~$\varphi(u) = \lim_\alpha \varphi_\alpha(u) = 1$.
        Thus indeed~$\Phi(E)$ is closed in~$\R^E$ as desired.
\end{proof}

\section{Yosida duality}
We are almost ready to state and prove Yosida duality.
First a notational matter:
recall the objects of~$\mathsf{CAURiesz}$
are pairs~$(E, u)$ of a uniformly complete Archimedean Riesz space~$E$
    together with a distinguished unit~$u \in E$.
For brevity we'll simply write~$E$ and denote the implicit
    distinguished unit by~$u_E$.
\begin{thm}[Yosida]
$C$ and~$\Phi$ are categorically dual:
\begin{enumerate}
\item
The map~$X \mapsto C(X)$
extends to a functor~$\mathsf{CHsdf} \to \mathsf{CAURiesz}^\partial$
taking~$\mathbb{1}$ as unit on~$C(X)$
and sending a continuous map~$f\colon X \to Y$
to the Riesz homomorphism~$C(f) \colon C(Y) \to C(X)$
given by~$C(f)(g) = g \after f$.

\item
    The map~$E \mapsto \Phi(E)$ extends to a functor~$\mathsf{CAURiesz}
        \to \mathsf{CHsdf}^\partial$
by mapping a Riesz homomorphism~$f \colon E \to E'$
to~$\Phi(f)\colon \Phi(E') \to \Phi(E)$
given by~$\Phi(f)(\varphi) = \varphi \after f$.

\item
We have~$C\Phi \cong \id$ (every uniformly complete Archimedean Riesz
space is naturally isomorphic to a continuous
    function space) and~$\Phi C \cong \id$
    (every compact Hausdorff space is naturally isomorphic to
    the spectrum of an Archimedean Riesz space)
    and so~$\mathsf{CHsdf}$ and~$\mathsf{CAURiesz}$
    are dually equivalent.
\end{enumerate}
\end{thm}
\begin{proof}
We first show 1 and 2.
\begin{enumerate}
\item
    We already saw~$C(X)$ is in~$\mathsf{CAURiesz}$
    (Example~\ref{exa:CXisCAU}
    and that~$\Phi(E)$ is in~$\mathsf{CHsdf}$
    (Lemma~\ref{lem:specishsdf}).
    Maps on arrows defined by
    pre-composition (such as~$C$ and~$\Phi$)
    are automatically functorial.
\item
    Clearly~$C(f)(g)\equiv g \after f$ is continuous
        as it is the composition of two continuous maps.
    To finish 1, we still have to check~$C(f)$ is a Riesz
        homomorphism.
    It is additive:~$C(f)(g+g')(x)
                = (g+g') \after f (x)
                = g(f(x)) + g'(f(x))
                = C(f)(g)(x) + C(f)(g)(x)$.
        In a similar fashion one checks~$C(f)$
         preserves scalar multiplication, $\vee$ and~$\wedge$.
    Unit-preservation is different, but
        also simple:~$C(f)(\mathbb{1})(x) = \mathbb{1} \after f (x)
                            = 1 = \mathbb{1}(x)$.

\item
    To wrap up 2 we need to check~$\Phi(f)$ is a continuous map
        for all~$f\colon E \to E'$.
    Let~$(\varphi_\alpha)_\alpha$ be a net in~$\Phi(E')$
        converging to~$\varphi$.
    In particular
        we have for any~$x\in E$ that~$\varphi_\alpha(f(x)) \to \varphi(f(x))$.
        Hence~$\Phi(f)(\varphi_\alpha) \equiv \varphi_\alpha \after f \to \varphi \after f \equiv \Phi(f)(\varphi)$
            in~$\Phi(E)$, as desired.
\end{enumerate}
Now we turn to the proof of~$C\Phi \cong \id$.
Let $E$ be in~$\textsf{CAURiesz}$.
\begin{enumerate}[resume]
\item
First we define what will turn out to be an isomorphism.        
For~$x \in E$, define~$\hat{x}$ (initially in~$\R^{\Phi(E)}$)
by~$\hat{x}(\varphi) := \varphi(x)$.
To prove~$\hat{x}$ is continuous,
assume~$\varphi_\alpha \to \varphi$ is a converging net
in~$\Phi(E)$. Then~$\hat{x}(\varphi_\alpha)
    = \varphi_\alpha(x) \rightarrow \varphi(x)
= \hat{x}(\varphi)$, thus~$\hat{x} \in C\Phi(E)$.
For the moment, write~$\hat{E}:=\{\hat{x};\ x \in E\}$.
\item
Note that~$(\hat{x} + \hat{y})(\varphi) = \varphi(x + y)
= \varphi(x) + \varphi(y) = \hat{x}(\varphi) + \hat{y}(\varphi)$
and thus~$\hat{x} + \hat{y} = \widehat{x+y} \in \hat{E}$.
The same reasoning applies to~$\cdot$, $\wedge$ and~$\vee$.
Thus~$x \mapsto \hat{x}$ is a Riesz-homomorphism
and~$\hat{E}$, being its image, is
a Riesz subspace of~$C\Phi(E)$.
\item
Suppose~$x \in E$ and~$x \neq 0$.
Then by Proposition~\ref{r-embedding-prop}, there is a~$\varphi \in \Phi(E)$
such that~$\varphi(x) \neq 0$. Thus~$\hat{x} \neq 0$. 
Thus~$x \mapsto \hat{x}$
is injective. Consequently~$E \cong \hat{E}$.
\item
Given~$\varphi, \psi \in \Phi(E)$ with~$\varphi \neq \psi$.
By definition there is an~$x \in E$
such that~$\varphi(x) \neq \psi(x)$.
Together with~$\hat{u}_E(\varphi) = \varphi(u_E) = \mathbb{1}$,
we conclude with the theorem of Stone-Weierstrass
(see Thm.~\ref{thm:stoneweierstrass}) that~$\hat{E}$
is $\|\ \|_{\mathbb{1}}$-dense in~$C\Phi(E)$. Since~$E \cong \hat{E}$ is
uniformly complete (i.e. $\|\ \|_{\mathbb{1}}$-complete)
we find~$C\Phi(E) \cong E$.

\item
We finish with naturality.
Assume~$f \colon E \to E'$ in~$\mathsf{CAURiesz}$, $x\in E$
and~$\varphi \in \Phi(E')$.
Expanding definitions we find~$C\Phi(f)(\hat{x})(\varphi)
= \hat{x}(\Phi(f) (\varphi)) = \hat{x}(\varphi \circ f)
= \varphi \circ f (x) = \widehat{f(x)}(\varphi)$.
Thus~$C\Phi \cong \id$.
\end{enumerate}
Now we prove the other direction: $\Phi C \cong \id$.
Let~$X \in \textsf{CHsdf}$ be given.
\begin{enumerate}[resume]
\item
Recall~$\delta_x \in \Phi C(X)$ defined by $\delta_x(f) \equiv f(x)$.
Write~$\Delta(X) := \{\delta_x;\  x\in X\}$.

\item
Suppose~$x_\alpha \rightarrow x$ in~$X$.
Then for~$f \in C(X)$ we have~$\delta_{x_\alpha}(f) = f(x_\alpha)
\rightarrow f(x) = \delta_{x}(f)$. Thus~$x \mapsto \delta_x$ is continuous.

\item
Let~$x,y\in X$ with~$x \neq y$ be given.
By Urysohn's Lemma there is an~$f\in C(X)$ such
that~$f(x) \neq f(y)$. Hence~$\delta_x \neq \delta_y$.
Thus~$x \mapsto \delta_x$ is injective.

\item
Assume~$\varphi \in \Phi C(X)$. By Proposition~\ref{phic-are-points}
there is an~$x \in X$ such that~$\varphi = \delta_x$.
Thus the map~$x \mapsto \delta_x$ is
surjective.
Together with the previous two points
    we see~$x \mapsto \delta_x$ is a continuous bijection
    between compact Hausdorff spaces and thus a homeomorphism.

\item
The demonstration of naturality is the same as in the previous Proposition.
Thus~$\Phi C \cong \id$ as desired.
\qedhere
\end{enumerate}        
\end{proof}

\subsection*{Acknowledgements}

Robert Furber kindly encouraged me to publish these notes.
I am grateful to Abraham Westerbaan
    for his suggestions.
\bibliography{main}{}
\bibliographystyle{alpha}

\end{document}